\documentclass[11pt]{article}
\usepackage[margin=1in]{geometry}
\usepackage{amsmath,amssymb,amsthm,mathtools,bm}
\usepackage{booktabs}
\usepackage{graphicx}
\usepackage{xcolor}
\usepackage[numbers,sort&compress]{natbib}
\usepackage[colorlinks=true,linkcolor=blue!50!black,citecolor=blue!50!black,urlcolor=blue!50!black]{hyperref}
\usepackage{enumitem}

\newtheorem{theorem}{Theorem}[section]
\newtheorem{proposition}[theorem]{Proposition}
\newtheorem{corollary}[theorem]{Corollary}

\theoremstyle{definition}
\newtheorem{assumption}{Assumption}

\newtheorem{example}[theorem]{Example}
\theoremstyle{remark}
\newtheorem{remark}[theorem]{Remark}

\newcommand{\R}{\mathbb{R}}
\newcommand{\E}{\mathbb{E}}
\newcommand{\tr}{\operatorname{tr}}
\newcommand{\dv}{\nabla\!\cdot\!}
\newcommand{\logit}{\operatorname{logit}}

\title{Regression Functions as Fields on Covariate Space:\\
Evolution Equations, the Hessian, and Identification}
\author{Nobuaki Tanaka}
\date{October 2026}

\begin{document}
\maketitle

\begin{abstract}
For each time $t$, the regression function $m_t(x)=\E[Y_t\mid X_t=x]$ of an outcome on contemporaneous covariates can be regarded as a field on covariate space. We study how this field evolves when the covariates of individual units move as a diffusion and the outcome attached to each unit changes according to its own dynamics. Under mild regularity we show that $m_t$ satisfies a second-order evolution equation whose diffusion term is $\tfrac12\tr(\Sigma\nabla^2 m_t)$, the trace of the Hessian of the regression function weighted by the instantaneous covariance $\Sigma$ of the covariate increments, whose transport term is the drift of the \emph{time-reversed} covariate process, and whose source term is the conditional mean rate of change of the outcome. The equation is a regression-theoretic reading of a classical transport law, and it yields a Feynman--Kac representation of $m_t$ along reversed covariate paths. We connect the Hessian appearing in the equation to the Hessian estimated by quadratic regression (under Gaussian design the best quadratic approximation has Hessian $\E[\nabla^2 m(X)]$ exactly) and to the leading bias of kernel smoothers. For binary outcomes governed by bistable replicator dynamics with imitation among units with similar covariates, the closed equation is of Nagumo type, its planar travelling fronts are exactly logistic in $(x,t)$, and a single logistic regression of the outcome on covariates and time identifies the ratio of diffusion to reaction and the direction of invasion. We then ask whether spread in covariate space reflects fluctuating unit characteristics or imitation among similar units. Imitation within a covariate neighbourhood enters the equation through exactly the same operator as covariate fluctuation, so the two are observationally equivalent in the regression field alone; panel data on covariates break the equivalence, because they identify the fluctuation covariance directly. We show that the two mechanisms cannot be separated on a travelling front, propose a split-sample instrumental-variables test that works on non-front (e.g.\ seeded) profiles, and evaluate its size and power in agent-based experiments. We also discuss identification under panel versus repeated cross-sectional designs.
\end{abstract}


%=====================================================================
\section{Introduction}\label{sec:intro}

A common way to add spatial structure to a dynamical model is to append a diffusion term to an ordinary differential equation, turning $\dot y = R(y)$ into a reaction--diffusion equation $\partial_t y = \kappa\Delta y + R(y)$ on a physical domain. This paper takes the same step with a different notion of space: the space of \emph{covariates} of a regression.

Fix a time $t$ and regress an outcome $Y_t$ on contemporaneous covariates $X_t\in\R^d$. The regression function
\[
  m_t(x)=\E[Y_t\mid X_t=x]
\]
is a field on covariate space. As $t$ varies, the covariates of individual units drift and fluctuate, and the outcome attached to each unit changes. The question we ask is: \emph{what equation does $m_t$ satisfy as a function of $(t,x)$?}

Our main result (Theorem~\ref{thm:main}) answers this under a diffusion model for covariates. The equation contains three terms with distinct interpretations:
\begin{enumerate}[label=(\roman*),leftmargin=2em]
  \item a diffusion term $\tfrac12\tr(\Sigma\nabla^2 m_t)$, where $\nabla^2 m_t$ is the Hessian of the regression function and $\Sigma$ the covariance rate of covariate increments; for $\Sigma=2\kappa I$ it is $\kappa\Delta m_t=\kappa\tr\nabla^2 m_t$;
  \item a transport term $\bar b\cdot\nabla m_t$, where $\bar b$ is the drift of the time-reversed covariate process; it contains a correction $\Sigma\nabla\log f_t$ involving the covariate density $f_t$;
  \item a source term equal to the conditional mean rate of change of the outcome.
\end{enumerate}

The diffusion coefficient is not chosen by the analyst: it is the covariance of covariate increments, which removes the arbitrariness of choosing a metric on a space whose coordinates carry different units. The equation is invariant under affine changes of covariates.

\paragraph{Relation to known results.}
Mathematically, the evolution equation is a regression-theoretic reading of classical facts. The product $q_t=m_tf_t$ obeys a conservation law, and $m_t$ plays the role of a mass fraction in species transport; the transport coefficient is the reverse-time drift of \citet{Anderson1982} and \citet{HaussmannPardoux1986}; and the same structure underlies the partial differential equation satisfied by the denoiser $\E[x_0\mid x_t=x]$ in diffusion generative models \citep{Daras2023}. Backward Kolmogorov equations for conditional expectations of future quantities are standard in option pricing \citep{LongstaffSchwartz2001} and in fate mapping of single cells \citep{Lange2022}. Forward equations for the \emph{density} of agents in a state space are central in heterogeneous-agent economics \citep{Gabaix2016,Achdou2022} and in trajectory inference from single-cell snapshots \citep{Weinreb2018,Schiebinger2019}, and replicator--mutator equations place diffusion on a continuous trait space \citep{Kimura1965,AlfaroCarles2017}. Section~\ref{sec:related} discusses these connections. To our knowledge, however, the regression function of an outcome on contemporaneous covariates has not been treated systematically as the unknown of such an equation, together with the resulting questions of interpretation, identification and estimation that are natural in regression analysis. That is the purpose of this paper.

\paragraph{Contributions.}
\begin{enumerate}[leftmargin=2em]
  \item An evolution equation for $m_t$ with explicit assumptions, including state-dependent diffusion and correlated outcome noise, and its Feynman--Kac representation along time-reversed covariate paths (Section~\ref{sec:main}).
  \item An interpretation of the Hessian term in regression terms: the Hessian of the best quadratic approximation under Gaussian design equals $\E[\nabla^2 m(X)]$, and the population Nadaraya--Watson target with Gaussian kernel solves the evolution equation with bandwidth playing the role of time (Section~\ref{sec:hessian}).
  \item Invariance of $[0,1]$ for closed equations with reaction terms vanishing at $0$ and $1$, via the sign of the Hessian at extrema (Section~\ref{sec:invariance}).
  \item For bistable replicator dynamics with imitation in covariate space, an exact logistic-regression identification of the diffusion--reaction ratio and of the invading strategy (Section~\ref{sec:replicator}).
  \item Imitation in covariate space acts through the same operator as covariate fluctuation; the two are identified separately from panel data, but not on a travelling front. We give a split-sample instrumental-variables test of imitation and evaluate it in agent-based experiments (Section~\ref{sec:imitation}).
  \item A discussion of identification under panel and repeated cross-sectional designs and a two-stage estimation procedure (Section~\ref{sec:estimation}), with simulations (Section~\ref{sec:simulations}).
\end{enumerate}

%=====================================================================
\section{Setting}\label{sec:setting}

Consider a large population of units. Unit $i$ carries a covariate process $X^i_t\in\R^d$ and an outcome process $Y^i_t\in\R$, $t\in[0,T]$. We describe a generic unit and drop the superscript.

\begin{assumption}[Covariate dynamics]\label{ass:cov}
$X_t$ solves the Itô equation
\[
  dX_t=b(t,X_t)\,dt+\sigma(t,X_t)\,dW_t ,
\]
with $\Sigma=\sigma\sigma^\top$. The coefficients are smooth, and $X_t$ has a density $f_t(x)$ that is $C^{1,2}$ and strictly positive on an open set $\mathcal X\subseteq\R^d$.
\end{assumption}

\begin{assumption}[Outcome dynamics]\label{ass:out}
The pair $(X_t,Y_t)$ is Markov with generator $\mathcal L_X+\mathcal A_x+\mathcal C$, where
\begin{itemize}[leftmargin=1.5em]
  \item $\mathcal L_X\varphi=b\cdot\nabla_x\varphi+\tfrac12\tr(\Sigma\nabla_x^2\varphi)$ acts on $x$;
  \item for each fixed $x$, $\mathcal A_x$ is the generator of the outcome dynamics acting on $y$ (an ordinary differential equation, a diffusion, or a jump process);
  \item $\mathcal C\psi=c(t,x,y)\cdot\nabla_x\partial_y\psi$ is a cross term with $c=d\langle X,Y\rangle_t/dt$; $c\equiv0$ when the noise of $Y$ is independent of $W$.
\end{itemize}
The function $\iota(y)=y$ is in the extended domain of $\mathcal A_x$, and
\[
  G(t,x):=\E\bigl[(\mathcal A_{X_t}\iota)(Y_t)\mid X_t=x\bigr],\qquad
  \bar c(t,x):=\E\bigl[c(t,X_t,Y_t)\mid X_t=x\bigr]
\]
are well defined. For an ordinary differential equation $dY_t=F(t,X_t,Y_t)\,dt$, $G(t,x)=\E[F(t,X_t,Y_t)\mid X_t=x]$.
\end{assumption}

\begin{assumption}[Regularity]\label{ass:reg}
$\E|Y_t|^2<\infty$; the map $(t,x)\mapsto q_t(x):=\E[Y_t\mid X_t=x]f_t(x)$ is $C^{1,2}$ on $(0,T)\times\mathcal X$; boundary terms in the integrations by parts below vanish (e.g.\ $\mathcal X=\R^d$ with sufficient decay, or zero-flux conditions on $\partial\mathcal X$).
\end{assumption}

The object of interest is the regression function $m_t(x)=q_t(x)/f_t(x)$ for $x\in\mathcal X$. The outcome need not be a deterministic function of the covariates: the conditional variance $\operatorname{Var}(Y_t\mid X_t)$ is arbitrary. We use the divergence of a matrix field row-wise, $(\dv\Sigma)_i=\sum_j\partial_j\Sigma_{ij}$.

%=====================================================================
\section{The evolution equation}\label{sec:main}

\begin{theorem}\label{thm:main}
Under Assumptions~\ref{ass:cov}--\ref{ass:reg}, on $(0,T)\times\mathcal X$,
\begin{equation}\label{eq:main}
  \partial_t m_t
  =\tfrac12\tr\bigl(\Sigma\,\nabla^2 m_t\bigr)
  +\bar b_t\cdot\nabla m_t
  +G_t
  -\frac{1}{f_t}\dv\bigl(f_t\,\bar c_t\bigr),
\end{equation}
where
\begin{equation}\label{eq:reversedrift}
  \bar b_t(x)=\Sigma(t,x)\nabla\log f_t(x)+(\dv\Sigma)(t,x)-b(t,x).
\end{equation}
\end{theorem}

\begin{proof}
Let $\varphi\in C_c^\infty(\mathcal X)$ and $\psi(x,y)=\varphi(x)y$. By Dynkin's formula,
\[
  \frac{d}{dt}\E[\varphi(X_t)Y_t]
  =\E\bigl[(\mathcal L_X\varphi)(X_t)\,Y_t+\varphi(X_t)(\mathcal A_{X_t}\iota)(Y_t)+c\cdot\nabla\varphi(X_t)\bigr].
\]
Conditioning on $X_t$ and writing the expectations as integrals against $f_t$ gives
\[
  \frac{d}{dt}\int\varphi\,q_t=\int(\mathcal L_X\varphi)\,q_t+\int\varphi\,f_tG_t+\int f_t\bar c_t\cdot\nabla\varphi .
\]
Integrating by parts, $q_t$ is a weak, hence by Assumption~\ref{ass:reg} classical, solution of
\begin{equation}\label{eq:q}
  \partial_t q_t=\mathcal L_X^{*}q_t+f_tG_t-\dv(f_t\bar c_t),\qquad
  \mathcal L_X^{*}u=-\dv(bu)+\tfrac12\textstyle\sum_{i,j}\partial_{ij}(\Sigma_{ij}u).
\end{equation}
The same argument with $\psi(x,y)=\varphi(x)$ gives the Fokker--Planck equation $\partial_t f_t=\mathcal L_X^{*}f_t$. Since $m_t=q_t/f_t$,
\[
  \partial_t m_t=\frac{\mathcal L_X^{*}(m_tf_t)-m_t\,\mathcal L_X^{*}f_t}{f_t}+G_t-\frac{1}{f_t}\dv(f_t\bar c_t).
\]
Expanding, $-\dv(bmf)+m\dv(bf)=-f\,b\cdot\nabla m$ and
\[
  \tfrac12\sum_{i,j}\bigl[\partial_{ij}(\Sigma_{ij}mf)-m\,\partial_{ij}(\Sigma_{ij}f)\bigr]
  =\tfrac12 f\tr(\Sigma\nabla^2 m)+\sum_{i,j}\partial_i m\,\partial_j(\Sigma_{ij}f).
\]
Dividing by $f$ and using $\partial_j(\Sigma_{ij}f)/f=(\Sigma\nabla\log f)_i+(\dv\Sigma)_i$ gives \eqref{eq:main}. (We also verified the algebra symbolically for $d=2$ with general $\Sigma(x)$ and $b(x)$.)
\end{proof}

\begin{remark}[The transport term is the reverse-time drift]\label{rem:reverse}
The vector field $\bar b_t$ in \eqref{eq:reversedrift} is the drift of the time reversal $\bar X_s=X_{T-s}$ of the covariate diffusion \citep{Anderson1982,HaussmannPardoux1986}. Equation~\eqref{eq:main} is therefore a backward Kolmogorov equation for the reversed process. This yields the following representation, which makes the equation intuitive: \emph{the regression function at $(t,x)$ averages the history of the units that are at $x$ at time $t$}.
\end{remark}

\begin{corollary}[Feynman--Kac representation]\label{cor:fk}
Suppose $dY_t=g(t,X_t)\,dt$ for a deterministic function $g$. Then
\[
  m_t(x)=\E\Bigl[m_0(X_0)+\int_0^t g(s,X_s)\,ds\ \Bigm|\ X_t=x\Bigr],
\]
and conditionally on $X_t=x$ the path $(X_s)_{s\le t}$ is the time-reversed diffusion started at $x$, with drift $\bar b$ and diffusion matrix $\Sigma$.
\end{corollary}

\begin{proof}
$Y_t=Y_0+\int_0^tg(s,X_s)ds$ and $\E[Y_0\mid X_0,X_t]=\E[Y_0\mid X_0]=m_0(X_0)$ when $Y_0$ is conditionally independent of the covariate noise given $X_0$. The second statement is the time-reversal theorem.
\end{proof}

\begin{remark}[Hessian and Laplacian]
For $\Sigma=2\kappa I$, the diffusion term is $\kappa\Delta m_t=\kappa\tr\nabla^2 m_t$. The identification of the Laplacian with the trace of the Hessian is a definition; what Theorem~\ref{thm:main} adds is that the weight matrix is the covariance rate of covariate increments. Under an affine change of covariates $x'=Ax+a$, we have $\Sigma'=A\Sigma A^\top$ and $\nabla'^2m=A^{-\top}(\nabla^2m)A^{-1}$, so $\tr(\Sigma'\nabla'^2m)=\tr(\Sigma\nabla^2m)$. No metric on covariate space has to be chosen.
\end{remark}

\begin{remark}[Mass-fraction form]
Equation~\eqref{eq:q} is a conservation law for $q_t=m_tf_t$, and $m_t$ is the ratio of a conserved quantity to the density $f_t$, exactly like the mass fraction of a species transported in a fluid of density $f_t$. In this sense, the conserved ``distribution'' on covariate space is $q_t$, not $m_t$; the term $\Sigma\nabla\log f_t$ arises from the division by $f_t$.
\end{remark}

\begin{example}[Special cases]\label{ex:special}
\begin{enumerate}[label=(\alph*),leftmargin=2em]
  \item \emph{Heat equation.} If $b=0$, $\Sigma=2\kappa I$, $\nabla f_t=0$ on the region considered, $G=0$ and $c=0$, then $\partial_t m_t=\kappa\Delta m_t$.
  \item \emph{Stationary Ornstein--Uhlenbeck covariates.} If $dX=-\Theta X\,dt+\sigma dW$ with stationary covariance $V$ satisfying $\Theta V+V\Theta^\top=\Sigma$, then $\nabla\log f=-V^{-1}x$ and $\bar b(x)=-V\Theta^\top V^{-1}x$, the drift of the reversed Ornstein--Uhlenbeck process. In one dimension $\bar b(x)=-\theta x$.
  \item \emph{Linear outcome dynamics.} If $dY_t=(\alpha(X_t)-\lambda Y_t)\,dt$, then $G_t=\alpha-\lambda m_t$ and \eqref{eq:main} is a closed linear equation for $m_t$.
\end{enumerate}
\end{example}

\paragraph{Closure.}
In general, $G_t$ depends on the conditional law of $Y_t$ given $X_t$ and not only on $m_t$. Equation~\eqref{eq:main} is closed in $m_t$ when the outcome dynamics are linear in $y$ (Example~\ref{ex:special}(c)) or, for binary outcomes, when transition rates depend on $x$ and on the local frequency $m_t(x)$ only (Section~\ref{sec:replicator}). Without closure, \eqref{eq:main} remains an exact identity linking estimable quantities, which is how it is used in Section~\ref{sec:estimation}.

%=====================================================================
\section{The Hessian in regression}\label{sec:hessian}

The Hessian $\nabla^2 m$ that drives \eqref{eq:main} also appears in two familiar places in regression: quadratic regression and kernel smoothing.

\subsection{Best quadratic approximation under Gaussian design}

Consider the quadratic model $m(x)\approx c+J^\top(x-\mu)+\tfrac12(x-\mu)^\top H(x-\mu)$, which is linear in the parameters $(c,J,H)$, so ordinary least squares on the features $(1,x_i,x_ix_j)$ estimates its $L^2(P_X)$ projection when the model is misspecified.

\begin{proposition}\label{prop:gauss}
Let $X\sim N(\mu,\Sigma_X)$ with $\Sigma_X$ nonsingular, and let $m$ be twice weakly differentiable with $\E|m(X)|^2$, $\E\|\nabla m(X)\|$ and $\E\|\nabla^2m(X)\|$ finite. The best quadratic approximation of $m$ in $L^2(P_X)$ has
\[
  J^*=\E[\nabla m(X)],\qquad H^*=\E[\nabla^2 m(X)].
\]
\end{proposition}

\begin{proof}
Let $u=X-\mu$. The functions $1$, $u$, and $u^\top Bu-\tr(B\Sigma_X)$ ($B$ symmetric) span the quadratic polynomials and are mutually orthogonal in $L^2(P_X)$ across degrees, because odd Gaussian moments vanish. Hence the projection can be computed degree by degree. For the linear part, Stein's identity $\E[m(X)u]=\Sigma_X\E[\nabla m(X)]$ and $\E[uu^\top]=\Sigma_X$ give $J^*=\E[\nabla m]$. For the quadratic part, the normal equations for $\tfrac12(u^\top Hu-\tr H\Sigma_X)$ read, for all symmetric $B$,
\[
  \E\bigl[m(X)\bigl(u^\top Bu-\tr B\Sigma_X\bigr)\bigr]
  =\tfrac12\E\bigl[(u^\top Hu-\tr H\Sigma_X)(u^\top Bu-\tr B\Sigma_X)\bigr]
  =\tr(\Sigma_XH\Sigma_XB),
\]
using $\operatorname{Cov}(u^\top Hu,u^\top Bu)=2\tr(H\Sigma_XB\Sigma_X)$. The second-order Stein identity gives the left side as $\tr(\Sigma_X\E[\nabla^2m]\Sigma_XB)$. Hence $\Sigma_XH^*\Sigma_X=\Sigma_X\E[\nabla^2 m]\Sigma_X$.
\end{proof}

Proposition~\ref{prop:gauss} extends the familiar statement that, under Gaussian design, the slope of the linear projection is the average gradient \citep{Stein1981,Brillinger1982}. In particular, the off-diagonal entry $H^*_{jk}$ of a misspecified quadratic regression estimates the average cross-partial derivative $\E[\partial_j\partial_k m(X)]$, which gives a model-free reading of interaction coefficients.

\subsection{Kernel smoothing as covariate diffusion}

Let $K$ be the standard Gaussian density on $\R^d$ and $Z\sim K$ independent of $(X,Y)$. The population target of the Nadaraya--Watson estimator with bandwidth $h$ is
\[
  \tilde m_h(x)=\frac{\int K_h(x-u)\,m(u)f(u)\,du}{\int K_h(x-u)\,f(u)\,du}=\E\bigl[Y\mid X+hZ=x\bigr].
\]
This is the regression function after the covariates have been diffused by Brownian motion for time $\tau=h^2$. Theorem~\ref{thm:main} with $b=0$, $\Sigma=I$, $G=0$ and $c=0$ shows that $\tau\mapsto\tilde m_{\sqrt\tau}$ solves
\[
  \partial_\tau\tilde m=\tfrac12\Delta\tilde m+\nabla\log\tilde f_\tau\cdot\nabla\tilde m ,
\]
where $\tilde f_\tau$ is the density of $X+\sqrt\tau Z$. Expanding at $\tau=0$ recovers the well-known leading bias $h^2\bigl(\tfrac12\Delta m+\nabla\log f\cdot\nabla m\bigr)$ of the Nadaraya--Watson estimator \citep{FanGijbels1996}. The local linear estimator removes the transport part and has leading bias $\tfrac12h^2\tr\nabla^2m$ \citep{RuppertWand1994}. Its leading bias therefore vanishes wherever $m$ is harmonic, and it is exactly conditionally unbiased when $m$ is linear.

%=====================================================================
\section{Invariance of the unit interval}\label{sec:invariance}

When $Y_t\in\{0,1\}$, $m_t=\Pr(Y_t=1\mid X_t)$ is automatically in $[0,1]$. A closed model $\partial_t p=\tfrac12\tr(\Sigma\nabla^2p)+v\cdot\nabla p+R(p)$ obtained by specifying $G=R(m)$ should preserve this property. It does whenever $R$ vanishes at the endpoints.

\begin{proposition}\label{prop:max}
Let $\Sigma$ be positive semidefinite and bounded, $v$ bounded, and $R$ locally Lipschitz with $R(0)=R(1)=0$. If $0\le p_0\le1$, then bounded classical solutions satisfy $0\le p_t\le1$ for all $t$.
\end{proposition}

\begin{proof}[Sketch]
At an interior maximum of $p_t$, $\nabla p_t=0$ and $\nabla^2p_t\preceq0$. Since $\Sigma\succeq0$, $\tr(\Sigma\nabla^2p_t)=\tr(\Sigma^{1/2}\nabla^2p_t\,\Sigma^{1/2})\le0$, so neither the diffusion nor the transport term can raise a maximum; the same holds for minima. The constants $0$ and $1$ are solutions, and the comparison principle for parabolic equations \citep{ProtterWeinberger1967} yields the claim.
\end{proof}

For $K\ge3$ types with frequencies $p^{(k)}$ and reaction terms $R_k$ satisfying $R_k=0$ when $p^{(k)}=0$ and $\sum_kR_k=0$, the probability simplex is invariant provided all components share the same diffusion matrix. Then $\sum_kp^{(k)}$ solves a linear equation with constant solution $1$, and each component stays nonnegative. With unequal diffusion coefficients, convex invariant regions of the reaction need not be preserved \citep{ChuehConleySmoller1977}. Two practical cautions apply. Explicit time discretizations need a step-size restriction (for the heat part, $\Delta t\lesssim\Delta x^2/(2d\kappa)$) to inherit the property. Unconstrained estimates of $m_t$, such as local quadratic regression, can leave $[0,1]$, which is one reason to work on the logit scale.

%=====================================================================
\section{Bistable replicator dynamics on covariate space}\label{sec:replicator}

\subsection{The closed equation}

Let each unit play strategy $A$ ($Y=1$) or $B$ ($Y=0$). In a two-strategy coordination game, the payoff difference between $A$ and $B$ is affine in the frequency $p$ of $A$ and changes sign at an interior point $a\in(0,1)$. Replicator dynamics \citep{HofbauerSigmund1998} then give the bistable reaction $R(p)=r\,p(1-p)(p-a)$. Suppose units imitate others whose covariates are close to their own, and that individual switching rates depend on $x$ only through the local frequency $p_t(x)=m_t(x)$. Then $G=R(m)$ and Theorem~\ref{thm:main} gives the closed equation
\begin{equation}\label{eq:nagumo}
  \partial_t p=\tfrac12\tr(\Sigma\nabla^2p)+v\cdot\nabla p+r\,p(1-p)(p-a),
\end{equation}
where $v=\bar b$ is the reverse-time drift. Non-local imitation with kernel $\kappa(\cdot)$, $\int\kappa(x-z)\bigl(p(z)-p(x)\bigr)dz$, leads to the same diffusion term at second order in the kernel width. With constant $\Sigma=2\kappa I$ and $v=0$, \eqref{eq:nagumo} is the Nagumo (bistable Allen--Cahn) equation \citep{McKean1970,AronsonWeinberger1978}.

\subsection{Travelling fronts are logistic}

\begin{proposition}\label{prop:wave}
Let $\Sigma$ and $v$ be constant, $n\in\R^d$ a unit vector, and $D_n=\tfrac12n^\top\Sigma n>0$. The function
\[
  p(t,x)=\Bigl(1+\exp\Bigl\{-\frac{n^\top x-ct}{w}\Bigr\}\Bigr)^{-1}
\]
solves \eqref{eq:nagumo} if and only if
\[
  w=\sqrt{2D_n/r},\qquad c=\sqrt{2D_nr}\,\bigl(a-\tfrac12\bigr)-v^\top n .
\]
\end{proposition}

\begin{proof}
Write $p=\phi(z)$ with $z=(n^\top x-ct)/w$ and $\phi$ the logistic function, so $\phi'=\phi(1-\phi)$ and $\phi''=\phi(1-\phi)(1-2\phi)$. Then $\partial_tp=-(c/w)\phi'$, $\tfrac12\tr(\Sigma\nabla^2p)=(D_n/w^2)\phi''$, and $v\cdot\nabla p=(v^\top n/w)\phi'$. Dividing \eqref{eq:nagumo} by $\phi(1-\phi)$ leaves an affine function of $\phi$ that vanishes identically if and only if $2D_n/w^2=r$ and $-(c+v^\top n)/w=(D_n/w^2)-ra$, which is equivalent to the stated $w$ and $c$.
\end{proof}

\begin{corollary}[Identification by logistic regression]\label{cor:logit}
On a travelling front,
\[
  \logit p(t,x)=\beta_0+\beta^\top x+\gamma t,\qquad
  \beta=\frac{n}{w},\quad \gamma=-\frac{c}{w}=-r\bigl(a-\tfrac12\bigr)+v^\top\beta .
\]
Hence, from the coefficients of a single logistic regression of $Y$ on $(X,t)$,
\begin{enumerate}[label=(\roman*),leftmargin=2em]
  \item $\beta/\|\beta\|$ is the normal of the front;
  \item $D_n/r=1/(2\|\beta\|^2)$; with isotropic $\Sigma=2\kappa I$, $\kappa/r=1/(2\|\beta\|^2)$;
  \item if $v=0$, $\gamma=-r(a-\tfrac12)$, so $\gamma>0$ if and only if $a<\tfrac12$, i.e.\ the risk-dominant strategy $A$ invades.
\end{enumerate}
Because the logistic model is correctly specified on the front, the maximum likelihood estimator of $(\beta_0,\beta,\gamma)$ inherits the usual consistency and asymptotic efficiency. The parameters $r$, $a$ and $\Sigma$ are not separately identified from a single front; only the combinations above are.
\end{corollary}

Bistable fronts are asymptotically stable and attract a large class of initial conditions \citep{FifeMcLeod1977}, which justifies fitting the model after a transient (Section~\ref{sec:simulations}).

\begin{remark}[Curved fronts and the Hessian of the index]
If the front is curved, one may add quadratic terms to the linear index, $\eta=\beta_0+\gamma t+\beta^\top x+\tfrac12x^\top H_\eta x$. The mean curvature of the level set $\{\eta=\text{const}\}$ is
\[
  -\nabla\cdot\frac{\nabla\eta}{\|\nabla\eta\|}=-\frac{\tr H_\eta-\nu^\top H_\eta\nu}{\|\nabla\eta\|},\qquad \nu=\frac{\nabla\eta}{\|\nabla\eta\|},
\]
the trace of the Hessian restricted to the tangent space. For the isotropic Allen--Cahn equation, the normal velocity of a slowly curved front is corrected by a multiple of this curvature. The Hessian thus enters twice: as the diffusion operator of the equation, and as the geometry of the logistic index.
\end{remark}

%=====================================================================
\section{Identification and estimation}\label{sec:estimation}

\subsection{What each design identifies}

\paragraph{Panel data.} If the same units are observed at times $t_0<t_1<\dots$, the covariate coefficients $b$ and $\Sigma$ are identified from covariate increments alone, as conditional mean and covariance rates of $X_{t_{k+1}}-X_{t_k}$ given $X_{t_k}$ \citep{ComteGenonCatalot2020}. The density $f_t$, the regression function $m_t$ and its derivatives are identified at each time. Equation~\eqref{eq:main} then identifies the source term,
\begin{equation}\label{eq:residual}
  G_t=\partial_tm_t-\tfrac12\tr(\Sigma\nabla^2m_t)-\bar b_t\cdot\nabla m_t+\frac{1}{f_t}\dv(f_t\bar c_t),
\end{equation}
without tracking outcomes over time. When outcomes are also tracked, $G_t$ is directly estimable as $\E[(Y_{t_{k+1}}-Y_{t_k})/\Delta t\mid X_{t_k}=x]$, and \eqref{eq:main} becomes a testable restriction linking separately estimated objects.

\paragraph{Repeated cross-sections.} If each time point has its own sample of units, only the marginal laws of $(X_t,Y_t)$ are identified. The density $f_t$ constrains $b$ and $\Sigma$ only through the Fokker--Planck equation. It is known from trajectory inference that snapshots do not identify divergence-free components of the drift \citep{Weinreb2018}. Equation~\eqref{eq:main} then identifies $G$ only up to the corresponding ambiguity in $\bar b$, unless the covariate dynamics are restricted parametrically (for instance a gradient drift or a known $\Sigma$).

\subsection{A two-stage procedure}

\begin{enumerate}[leftmargin=2em]
  \item \textbf{Covariate dynamics.} Estimate $b$ and $\Sigma$ by local regression of increments on $X$, and $f_t$ and $\nabla\log f_t$ by kernel methods; form $\hat{\bar b}_t$ via \eqref{eq:reversedrift}.
  \item \textbf{Regression field.} At each time, estimate $m_t$, $\nabla m_t$ and $\nabla^2m_t$ by local polynomial regression of degree three (odd degree is recommended for second derivatives; \citealp{FanGijbels1996}). On a binary outcome, a local logistic fit is preferable (Section~\ref{sec:invariance}).
  \item \textbf{Source and closure.} Compute the residual \eqref{eq:residual}. Either regress it on a dictionary of candidate reaction terms $\{m,\,m(1-m),\,m^2(1-m),\,u_t,\dots\}$ with sparsity penalties, in the spirit of PDE discovery \citep{Brunton2016,Rudy2017}, or fit a parametric closure by least squares.
\end{enumerate}

Three issues deserve attention. First, the $\nu$-th derivative of a $p$-smooth regression function can be estimated at best at rate $n^{-(p-\nu)/(2p+d)}$ \citep{Stone1982}, so Hessian estimates are noisy and the dimension of the covariate space should be small. This argues for reducing covariates to the few directions that matter, preferably with supervised dimension reduction rather than principal components. Second, the regressors in step 3 are estimated, which induces attenuation bias in the coefficients of diffusion-type terms. Sample splitting and integrated (weak) forms of \eqref{eq:main} over time and against test functions mitigate both this bias and the noise of finite-difference time derivatives. Third, \eqref{eq:main} holds where $f_t>0$; near the edge of the support, $\nabla\log f_t$ is large and boundary conditions matter. Although $f_t$ may be positive everywhere for a diffusion, the region where $m_t$ is \emph{estimable} grows only as fast as the data spread, roughly like $\sqrt t$. A formal asymptotic theory for the composite estimator is left for future work.

%=====================================================================
\section{Simulations}\label{sec:simulations}

\subsection{Checking the evolution equation}

We simulated $N=2\times10^6$ units with one-dimensional stationary Ornstein--Uhlenbeck covariates $dX=-\theta X\,dt+\sqrt s\,dW$ ($\theta=s=1$), outcomes $dY=(\sin 2X-\lambda Y)\,dt$ with $\lambda=0.5$, and $Y_0=\cos X_0+0.3\varepsilon$, using Euler--Maruyama steps of size $10^{-3}$. At $t=0.48,0.50,0.52$ we estimated $m_t$, $m_t'$ and $m_t''$ by local quadratic regression (Gaussian kernel, $h=0.12$) on a grid of $25$ points in $[-1.2,1.2]$. We compared the central difference $\partial_t\hat m_t$ with the right-hand side of \eqref{eq:main}, for which $\bar b(x)=-\theta x$ (Example~\ref{ex:special}(b)) and $G=\sin 2x-\lambda m$.

Figure~\ref{fig:master} shows close agreement. Over $|x|\le1$, the root mean square discrepancy is $0.025$, compared with a root mean square of $0.44$ for $\partial_t\hat m_t$ itself. Omitting the density correction $\Sigma\nabla\log f$, i.e.\ transporting with the forward drift $-b$, raises the discrepancy to $0.34$. Dropping transport altogether gives $0.18$. The correction term is therefore essential: the regression function is transported along the reversed, not the forward, covariate dynamics.

\begin{figure}[t]
  \centering
  \includegraphics[width=0.66\textwidth]{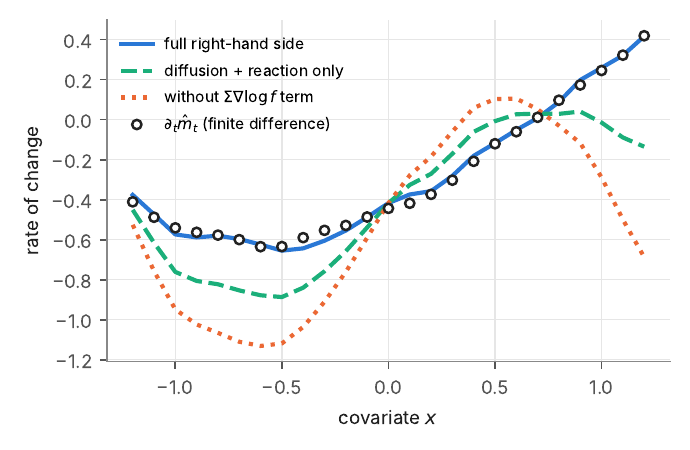}
  \caption{Time derivative of the estimated regression function (circles) and three versions of the right-hand side of \eqref{eq:main}: full (solid), without the transport term (dashed), and with the forward drift instead of the reverse-time drift (dotted). Ornstein--Uhlenbeck covariates, linear outcome dynamics, $N=2\times10^6$.}
  \label{fig:master}
\end{figure}

\subsection{Logistic identification of a bistable front}

We solved \eqref{eq:nagumo} with $d=2$, $\Sigma=2\kappa I$, $\kappa=0.5$, $r=1$, $a=0.3$, $v=0$ and front normal $n=(0.6,0.8)$. The theoretical values are $\kappa/r=0.5$ and $r(a-\tfrac12)=-0.2$. The initial condition was a steep, non-logistic planar step with values in $[0.05,0.95]$. We used an explicit finite-difference scheme in the normal coordinate with a stable time step. The solution stayed in $[0,1]$, as Proposition~\ref{prop:max} requires, and converged to the logistic front of Proposition~\ref{prop:wave} (Figure~\ref{fig:nagumo}). At each time on a grid of spacing $0.5$, we drew $400$ covariate vectors uniformly from $[-6,6]^2$ and binary outcomes from $p(t,X)$, and fitted a single logistic regression of $Y$ on $(X,t)$. Table~\ref{tab:nagumo} reports means and standard deviations over $300$ replications.

\begin{figure}[t]
  \centering
  \includegraphics[width=0.66\textwidth]{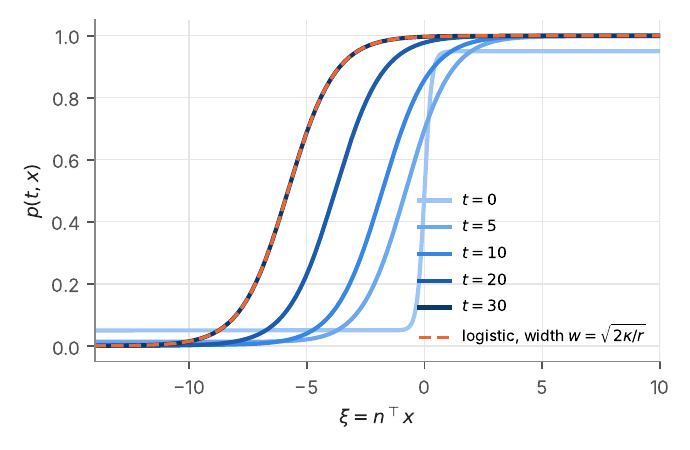}
  \caption{Solution of the bistable equation \eqref{eq:nagumo} along the front normal, from a non-logistic initial step to the logistic travelling front (dashed) of Proposition~\ref{prop:wave}.}
  \label{fig:nagumo}
\end{figure}

\begin{table}[t]
  \centering
  \caption{Logistic-regression estimates of the identified combinations (mean, with standard deviation in parentheses, over 300 replications).}
  \label{tab:nagumo}
  \begin{tabular}{lcccc}
    \toprule
    & $\kappa/r$ & $r(a-\frac12)$ & $n_1$ & $n_2$\\
    \midrule
    Truth & $0.500$ & $-0.200$ & $0.600$ & $0.800$\\
    Window $t\in[15,25]$ & $0.499\ (0.026)$ & $-0.200\ (0.014)$ & $0.601\ (0.011)$ & $0.799\ (0.008)$\\
    Window $t\in[0,10]$ & $0.487\ (0.024)$ & $-0.183\ (0.011)$ & $0.595\ (0.007)$ & $0.804\ (0.005)$\\
    \bottomrule
  \end{tabular}
\end{table}

After the transient, the estimates are centred on the truth. A window that includes the transient biases the invasion-speed combination $r(a-\tfrac12)$ by about $8\%$, roughly one and a half standard deviations, because the front has not yet reached its asymptotic speed.

%=====================================================================
\section{Fluctuating characteristics or imitation?}\label{sec:imitation}

When adoption of a strategy appears to spread across covariate space, two mechanisms are plausible. Either the characteristics of units fluctuate, so that adopters drift into new regions of covariate space, or units imitate others with similar characteristics. This section shows that the two mechanisms act through the same operator on the regression field, explains what separates them, and gives a test.

\subsection{Imitation acts like covariate fluctuation}

For a symmetric positive semidefinite matrix $V$, define the operator
\[
  \mathcal S_V[p]=\tfrac12\tr(V\nabla^2p)+V\nabla\log f\cdot\nabla p .
\]
With constant $\Sigma$, the covariate part of \eqref{eq:main} is $\mathcal S_\Sigma[m]-b\cdot\nabla m$.

Suppose that, in addition to the reaction $R$, each unit at rate $\nu_I$ draws another unit with probability proportional to $K(x-z)$, where $x$ and $z$ are their covariates, and copies its strategy. Here $K$ is a symmetric kernel with covariance $V$. The unit is drawn among existing units, so the local frequency it sees is
\[
  \tilde p(x)=\frac{\int K(x-z)\,p(z)f(z)\,dz}{\int K(x-z)\,f(z)\,dz},
\]
and the contribution of imitation to the source term is $G_I=\nu_I(\tilde p-p)$.

\begin{proposition}\label{prop:imit}
As $V\to0$, $\tilde p-p=\mathcal S_V[p]+O(\|V\|^2)$. Consequently, with constant $\Sigma$ and $b$ linear,
\begin{equation}\label{eq:imit}
  \partial_tp=\mathcal S_{\Sigma+\Sigma_I}[p]-b\cdot\nabla p+R(p)+O(\nu_I\|V\|^2),\qquad \Sigma_I=\nu_IV .
\end{equation}
\end{proposition}

\begin{proof}
For a smooth $g$, $(K*g)(x)=g(x)+\tfrac12\tr(V\nabla^2g(x))+O(\|V\|^2)$. Apply this to $g=pf$ in the numerator and $g=f$ in the denominator, and expand the ratio:
\[
  \tilde p=p+\frac{\tr(V\nabla^2(pf))-p\tr(V\nabla^2f)}{2f}+O(\|V\|^2)=p+\tfrac12\tr(V\nabla^2p)+\frac{\nabla p^\top V\nabla f}{f}+O(\|V\|^2).
\]
Add $\nu_I(\tilde p-p)$ to the source term of \eqref{eq:main}. (We verified the expansion symbolically in one dimension.)
\end{proof}

Proposition~\ref{prop:imit} says that imitation among similar units is \emph{observationally equivalent} to an extra covariate fluctuation with covariance $\Sigma_I$, including the density correction $V\nabla\log f$. The regression field alone, even observed perfectly, identifies only $\Sigma+\Sigma_I$. Panel data break the equivalence: the covariance $\Sigma$ of covariate increments is identified from the covariates alone (Section~\ref{sec:estimation}), so $\Sigma_I$ is identified as the excess diffusion of the regression field over $\Sigma$. Hence the null hypothesis of no imitation is
\[
  H_0:\ \Sigma_I=0 .
\]

\subsection{No separation on a travelling front}

The test needs one more condition. Take a logistic front $p=\phi(n^\top x/w)$ with uniform density. Then $\tfrac12n^\top Vn\,\partial^2_{nn}p=\frac{n^\top Vn}{2w^2}\bigl\{p(1-p)-2p^2(1-p)\bigr\}$, which lies in the span of the cubic reaction basis $\{p(1-p),\,p^2(1-p)\}$. On a front, therefore, any amount of diffusion can be traded against a change in the reaction coefficients. Even with $\Sigma$ known, $\Sigma_I$ is not identified when the reaction is an unknown bistable cubic: only the combinations of Corollary~\ref{cor:logit} are. Identification requires profiles in which $\nabla^2p$ is not a function of $p$, such as early transients or irregular seeding of initial adopters. This has a design implication: an experiment that seeds adopters at several places in covariate space is informative about imitation, and observing a mature front is not.

\subsection{A split-sample instrumental-variables test}

Second derivatives of an estimated regression field are noisy. They also carry smoothing bias of a different order from the field itself. We therefore work with $q_t=m_tf_t$ and $f_t$, whose Gaussian-kernel estimates $q_h=K_h*\hat q$ and $f_h=K_h*\hat f$ are linear in the data, so that derivatives commute with smoothing. Multiplying \eqref{eq:imit} by $f$ and using $f\,\mathcal S_V[m]=\tfrac12\{\tr(V\nabla^2q)-m\tr(V\nabla^2f)\}$ gives the $q$-form
\begin{equation}\label{eq:qform}
  \partial_tq=\mathcal L_X^*q+\tfrac12\bigl\{\tr(\Sigma_I\nabla^2q)-m\tr(\Sigma_I\nabla^2f)\bigr\}+f\,R(m).
\end{equation}
For the linear drift $b=-\Theta x$, convolution with a Gaussian kernel of covariance $h^2I$ satisfies $K_h*\nabla\!\cdot(xq)=\nabla\!\cdot(x\,q_h)+h^2\Delta q_h$ exactly. We use this to apply $\mathcal L_X^*$ to the smoothed field without bias. The test proceeds as follows.
\begin{enumerate}[leftmargin=2em]
  \item Estimate $\Theta$ and $\Sigma$ from covariate increments; for Ornstein--Uhlenbeck covariates, use the exact autoregressive transition.
  \item Split the units at random into halves $\mathcal A$ and $\mathcal B$. In each half and at each observation time, compute $q_h$, $f_h$ and their derivatives on a grid, the dependent variable $Z=\partial_tq_h-\hat{\mathcal L}_X^*q_h$ (central differences in time), the imitation regressor $\tfrac12\{\tr\nabla^2q_h-m_h\tr\nabla^2f_h\}$ (isotropic $V$; one regressor per entry otherwise), and the reaction regressors $f_hm_h^2(1-m_h)$ and $f_hm_h(1-m_h)$, where $m_h=q_h/f_h$.
  \item Regress $Z$ from half $\mathcal A$ on the regressors from $\mathcal A$, using the regressors from $\mathcal B$ as instruments. Swap the roles of the halves and average the two estimates. Because estimation noise is independent across halves, this removes the attenuation that noisy second derivatives would otherwise cause.
  \item Obtain a standard error by resampling units with replacement within each half. Reject $H_0$ when $\hat\Sigma_I/\widehat{\mathrm{se}}$ exceeds the one-sided normal critical value.
\end{enumerate}
The coefficients on the reaction regressors estimate $r$ and $-ra$, so $r$ and $a$ are recovered together with $\Sigma_I$.

\subsection{Agent-based experiments}

We generated synthetic panel data from an agent-based model with $N=10^5$ units and one-dimensional covariates. Covariates follow a stationary Ornstein--Uhlenbeck process with $\theta=0.25$ and $\Sigma=0.2$ (stationary variance $0.4$), simulated exactly with time step $0.01$; in one scenario they are frozen ($\Sigma=0$). Strategies evolve by two independent mechanisms.
\begin{itemize}[leftmargin=1.5em]
  \item \emph{Payoff-driven revision.} At rate $\nu_R=2$, a unit compares itself with a unit drawn uniformly among those within $\pm0.02$ in covariate space. If their strategies differ, it adopts the other's with probability $\tfrac12\pm\kappa(\tilde p-a)$, with $\kappa=0.25$ and $a=0.3$. To leading order this gives the bistable reaction $r\,p(1-p)(p-a)$ with $r=2\nu_R\kappa=1$; the window is so narrow that its own smoothing contributes $\nu_R(0.02)^2/6\approx10^{-4}$ to the diffusion.
  \item \emph{Imitation.} At rate $\nu_I$, a unit copies a unit drawn uniformly among those within $\pm0.3$, so $V=0.3^2/3=0.03$ and $\Sigma_I=0.03\,\nu_I$.
\end{itemize}
Initial adopters are seeded irregularly, $p_0(x)=0.5+0.3\sin(2\pi x/1.4+0.4)+0.15\cos(2\pi x/0.9)$ (clipped to $[0.03,0.97]$). All units are observed every $0.1$ time units up to $T=3$. Figure~\ref{fig:field} shows that the regression fields with and without imitation look much alike. We used bandwidth $h=0.1$, a grid of $23$ points on $[-1.1,1.1]$, observation times $t\in[0.2,2.9]$, and $49$ bootstrap draws. Each scenario was replicated $30$ times.

\begin{figure}[t]
  \centering
  \includegraphics[width=0.95\textwidth]{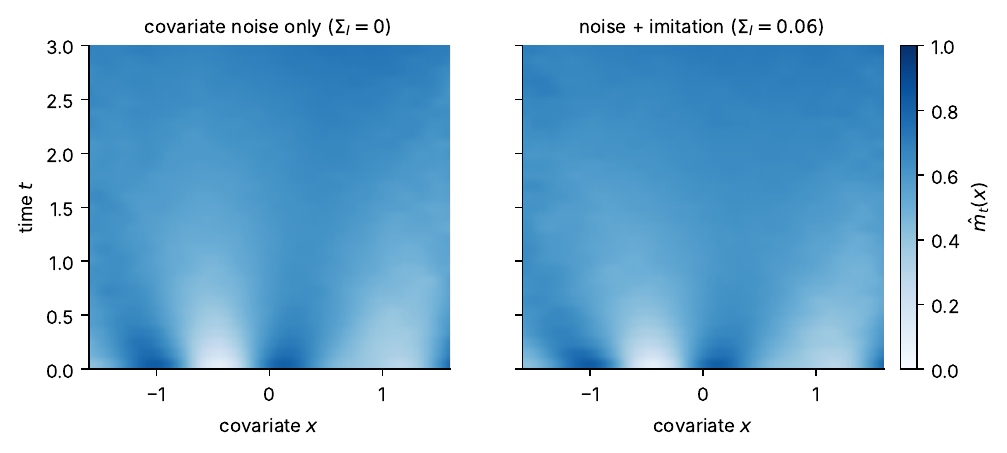}
  \caption{Estimated regression fields $\hat m_t(x)$ from the agent-based model with irregular seeding, without imitation (left) and with imitation of total strength $\Sigma_I=0.06$ (right). Covariate fluctuation is $\Sigma=0.2$ in both panels.}
  \label{fig:field}
\end{figure}

\begin{table}[t]
  \centering
  \caption{Split-sample IV test of imitation in the agent-based model ($30$ replications per scenario, $N=10^5$). Columns: true covariate fluctuation $\Sigma$ and imitation $\Sigma_I$; mean and Monte Carlo standard deviation of $\hat\Sigma_I$; mean bootstrap standard error; rejection rate of the one-sided $5\%$ test of $H_0:\Sigma_I=0$; mean estimates of $r$ (true $1$) and $a$ (true $0.3$). Scenario C$_{\rm narrow}$ uses an imitation window of $\pm0.2$ with rate $\nu_I=4.5$, giving the same $\Sigma_I$ as C with a smaller $V$.}
  \label{tab:imit}
  \small
  \begin{tabular}{lcccccccc}
    \toprule
    Scenario & $\Sigma$ & $\Sigma_I$ & mean $\hat\Sigma_I$ & MC sd & boot.\ se & reject & $\hat r$ & $\hat a$\\
    \midrule
    A & $0.2$ & $0$ & $0.001$ & $0.007$ & $0.010$ & $0.00$ & $1.02$ & $0.30$\\
    B$_{\rm weak}$ & $0.2$ & $0.03$ & $0.025$ & $0.007$ & $0.012$ & $0.73$ & $0.89$ & $0.25$\\
    B & $0.2$ & $0.06$ & $0.056$ & $0.008$ & $0.016$ & $1.00$ & $0.95$ & $0.28$\\
    C & $0$ & $0.06$ & $0.050$ & $0.002$ & $0.001$ & $1.00$ & $0.79$ & $0.23$\\
    C$_{\rm narrow}$ & $0$ & $0.06$ & $0.055$ & $0.003$ & $0.001$ & $1.00$ & $0.83$ & $0.26$\\
    \bottomrule
  \end{tabular}
\end{table}

Table~\ref{tab:imit} summarizes the results. The covariance of covariate increments is recovered essentially exactly ($\hat\Sigma=0.200$ when $\Sigma=0.2$). Under the null of no imitation (scenario A), $\hat\Sigma_I$ is centred at zero (mean $0.0013$), and the test rejected in none of the $30$ replications. The bootstrap standard error ($0.010$) exceeds the Monte Carlo standard deviation ($0.007$), so the test is conservative in this design. With covariate fluctuation and imitation both present, the test rejected in all replications for $\Sigma_I=0.06$ and in $73\%$ of replications for $\Sigma_I=0.03$, which is only $15\%$ of the covariate fluctuation. In those scenarios $\hat\Sigma_I$ is slightly below the truth (by about $0.004$--$0.005$). Without covariate fluctuation (scenario C), imitation is detected without exception, but $\Sigma_I$ is underestimated by about $17\%$; the bootstrap standard error is also smaller than the Monte Carlo standard deviation there. Reducing the kernel variance from $0.03$ to $0.013$ at a fixed $\Sigma_I$ (scenario C$_{\rm narrow}$) roughly halves the bias (to about $8\%$). This is consistent with the $O(\nu_I\|V\|^2)$ remainder of Proposition~\ref{prop:imit}: the imitation window is wide compared with the scale of the seeded profile, and without covariate fluctuation the field stays sharper.

\paragraph{Fronts versus seeded profiles.}
To illustrate the non-identification on fronts, we repeated scenario B $20$ times with the seeded profile and $20$ times with an initial logistic front ($p_0(x)=\phi((x-0.3)/0.25)$), all else equal. With seeding, $\hat\Sigma_I$ had mean $0.059$ and standard deviation $0.012$. With the front, the mean was $0.067$ but the standard deviation was $0.142$, about $12$ times larger, and the standard deviation of $\hat r$ rose from $0.22$ to $2.13$. Near a front, the diffusion and reaction regressors are almost collinear, as the analysis above predicts.

%=====================================================================
\section{Related work}\label{sec:related}

\paragraph{Transport, time reversal and Feynman--Kac.}
The conservation law \eqref{eq:q} and its mass-fraction form are standard in continuum mechanics. The reverse-time drift \eqref{eq:reversedrift} is due to \citet{Anderson1982} and \citet{HaussmannPardoux1986}. Conditional expectations of future quantities given current states solve backward Kolmogorov equations, which underlie regression-based Monte Carlo methods in finance \citep{LongstaffSchwartz2001}. Our equation is the forward-in-time counterpart for conditional expectations of current outcomes, and it is a backward equation for the reversed process (Corollary~\ref{cor:fk}).

\paragraph{Diffusion generative models.}
In diffusion models, the denoiser $\E[x_0\mid x_t=x]$ is the regression function of the clean signal on the noisy state. Its relation to the score is Tweedie's formula \citep{Efron2011}, and \citet{Daras2023} show that it satisfies a partial differential equation in $(t,x)$, which they use as a training constraint. This is the special case of Theorem~\ref{thm:main} with $G=0$ and a designed noising process. Our setting differs in that the covariate dynamics are estimated from data and the outcome has its own dynamics.

\paragraph{Single-cell dynamics.}
Fate probabilities in CellRank \citep{Lange2022} are conditional expectations of future labels given the current state, computed as absorption probabilities, i.e.\ discrete backward equations. Trajectory inference from snapshots \citep{Schiebinger2019} models the density of cells in expression space; \citet{Weinreb2018} delineate what snapshots can identify, which informs Section~\ref{sec:estimation}.

\paragraph{Heterogeneous agents and evolutionary dynamics.}
Forward equations for distributions over individual states are central in macroeconomics with heterogeneous agents \citep{Gabaix2016,Achdou2022}. Replicator--mutator equations place diffusion on a continuous trait space \citep{Kimura1965,AlfaroCarles2017}. Bistable reaction--diffusion fronts are classical \citep{McKean1970,FifeMcLeod1977,AronsonWeinberger1978}. These works study densities or trait distributions; we study regression functions over covariates.

\paragraph{Regression with differential operators.}
Spatial regression with PDE regularization \citep{Sangalli2021} uses differential operators on geographic domains as penalties. Sparse identification of dynamics \citep{Brunton2016,Rudy2017} selects PDE terms by regression. Our Section~\ref{sec:estimation} uses the latter idea on covariate space, with the diffusion and transport parts pinned down by covariate dynamics. The interpretation of misspecified regression coefficients as projections \citep{Buja2019} and as average derivatives \citep{Stein1981,Brillinger1982} motivates Section~\ref{sec:hessian}.

%=====================================================================
\section{Discussion}\label{sec:discussion}

Treating a cross-sectional regression function as a field on covariate space turns familiar regression objects, namely gradients, Hessians, smoothing bias and logistic indices, into the ingredients of an evolution equation whose coefficients have a mechanistic meaning. Several directions remain open. They include an asymptotic theory for the two-stage estimator with estimated Hessians; extensions to non-diffusive covariate dynamics such as jumps and births and deaths of units; closures beyond the mean-field case; endogenous covariate dynamics in which $b$ depends on outcomes, so that \eqref{eq:main} no longer closes in $m_t$; and applications to repeated cross-sections, where identification hinges on restrictions on the covariate drift. For the imitation test, a theory for the bootstrap with intrinsic population noise, higher-order corrections for wide imitation kernels, and multivariate covariates with anisotropic $V$ are natural next steps.

\section*{Statement on contributions and the use of AI}
This work grew out of a conversation between the author and Claude (Anthropic), an AI system, held on 27--28 October 2026 (version~1), and continued for version~2. The conversation began with an exercise on classical multidimensional scaling and moved on to regression. Because the boundary between human and machine contributions matters for a claim of priority, we record it here as precisely as the conversation allows. AI systems are not authors, and the author takes responsibility for all content.

\paragraph{Ideas originating with the author.}
\begin{enumerate}[leftmargin=2em]
  \item Going beyond the routine of reducing dimension by principal components and then fitting a linear regression, by also including a second-order (Hessian) term in the regression function so that interactions such as ``population $\times$ attention'' in game sales can be captured.
  \item Questions about the properties of standard estimators when the family of distributions is restricted to the linear regression model, and about why linear regression is legitimate when the true distribution lies outside the model. These questions arose in the author's prior notes, which organize regression by loss, family of distributions, regression class and comparison principle.
  \item Treating the trace of the Hessian as a Laplacian and forming, at each time, an evolution equation of the form $\partial_tY_t=\Delta Y_t+f(u)$.
  \item The central idea of this paper: take the \emph{covariates} of the regression at a fixed time as the space, regard the regression function at each time as a field on that space even though $\operatorname{Var}(\varepsilon\mid X)\neq0$, let time vary, and study the resulting heat-type equation, in which the Hessian appears.
  \item The observation that adding a diffusion term to a time-only ordinary differential equation is standard, and that taking the covariates as the space makes it natural for the range of covariates to spread over time.
  \item The motivating application: bistable replicator dynamics with a Hessian-based diffusion term on covariate space, combined with logistic regression of success rates. Also the recollection that the diffusion term does not push the frequency $p$ above one.
  \item The suggestion that, after multidimensional scaling, polynomial approximation with AIC-based selection of effective parameters can approach the truth in dimension $d>2$.
  \item The decision to write up the idea and post it to arXiv.
  \item (Version 2) The observation that adding gradient-flow components in covariate space to replicator dynamics yields an evolution equation for $\E[Y_t\mid\sigma(X_t)]$, and the request to test, on experimental data, whether spread in covariate space reflects fluctuating unit characteristics or imitation among similar units.
\end{enumerate}

\paragraph{Contributions of the AI system.}
\begin{enumerate}[leftmargin=2em]
  \item Formalizing the author's idea. This covers the setting (Assumptions~\ref{ass:cov}--\ref{ass:reg}), the derivation of Theorem~\ref{thm:main} (including the density-gradient term, state-dependent $\Sigma$ and the cross term), the observation that the transport coefficient is the reverse-time drift together with Corollary~\ref{cor:fk}, and the remark that weighting the Hessian by the covariance of covariate increments removes the need to choose a metric.
  \item Proposition~\ref{prop:gauss} and its proof; the connection between kernel smoothing and covariate diffusion in Section~\ref{sec:hessian}; the proof sketch of Proposition~\ref{prop:max} and the caveat for several types.
  \item Proposition~\ref{prop:wave}, Corollary~\ref{cor:logit} and the curvature remark, which turn the author's replicator--logistic idea into exact identification formulas.
  \item Critical comments that shaped the framing. Among them: the identity $\Delta=\tr\nabla^2$ is a definition and not a finding; principal components or classical scaling with all dimensions retained do not change a polynomial fit; polynomial approximation and AIC approach the regression function rather than the full distribution, and only under stated conditions.
  \item A literature search and novelty assessment, which identified the closely related work cited in Section~\ref{sec:related} and led to the framing of the contribution in Section~\ref{sec:intro}.
  \item Section~\ref{sec:estimation}, the symbolic verification of the derivations, the simulation code, the figures, and the drafting of this English manuscript.
  \item (Version 2) Framing the two mechanisms as a testable hypothesis; Proposition~\ref{prop:imit} (observational equivalence of imitation and covariate fluctuation) and the non-identification on travelling fronts; the $q$-form split-sample instrumental-variables test; and the agent-based experiments. The experimental data are synthetic, generated by the agent-based model of Section~\ref{sec:imitation}.
\end{enumerate}

\paragraph{Circumstances of writing.}
The author was skeptical of AI systems. Having disclosed the idea in the conversation, the author was concerned that it could be appropriated, and decided to post to arXiv promptly in order to establish a public record that the idea was the author's own and to protect both parties against later questions of priority. The author then instructed the AI system to typeset the manuscript immediately, and the present version was produced within the same conversation. Readers should judge the originality claims in Section~\ref{sec:intro} in light of the record above.

%=====================================================================

\end{document}